\documentclass[11pt]{amsart}
\usepackage{amsmath,amsfonts,xargs}
\usepackage{graphicx} 
\usepackage{amsmath, amsthm, amssymb, amsfonts}
\usepackage{tkz-tab}

\title{On the volume of sums of  Anti-blocking bodies}
\author{Auttawich Manui, Cheikh Saliou Ndiaye 
 and Artem Zvavitch}
 \thanks{Authors are supported in part by the U.S. National Science Foundation Grant DMS-1101636,
the United States - Israel Binational Science Foundation (BSF) Grant 2018115, and the Thailand Development and Promotion of Science and Technology talent project (DPST) } 


\date{}

\begin{document}

\newcommand{\vol}{\mathrm{vol}}
\newcommand{\R}{{\mathbb R}}
\newcommand{\G}[1][n]{\mathcal{G}^c_{#1}}
\NewDocumentCommand{\GK}{ O{n} O{k}}{\mathcal{G}^c_{#1,#2}}
\newcommand{\GM}[1][i]{\Gamma \left(1+\frac{#1}{q}\right)}
\

\newtheorem{theorem}{Theorem}
\newtheorem{lemma}[theorem]{Lemma}
\newtheorem{claim}[theorem]{Claim}
\newtheorem{remark}[theorem]{Remark}
\newtheorem{fact}[theorem]{Fact}
\newtheorem{corollary}[theorem]{Corollary}
\newtheorem{defi}[theorem]{Definition}
\newtheorem{prop}[theorem]{Proposition}
\newtheorem{ques}[theorem]{Question}
\newcommand{\Span}{\operatorname{span}}
\newcommand{\Int}{\operatorname{int}}
\newcommand{\Conv}{\operatorname{conv}}

\begin{abstract} 
We study inequalities on the volume of Minkowski sum in the class of anti-blocking bodies. We prove   analogues of Pl\"unnecke-Ruzsa type inequality  and  V. Milman inequality on  the concavity of the ratio of volumes  of bodies and their projections.  
We also study Firey $L_p-$sums of anti-blocking bodies and prove Pl\"unnecke-Ruzsa type inequality; V. Milman inequality  and Roger-Shephard inequality.  The sharp constants are provided in all of those inequalities, for the class of anti-blocking bodies. Finally, we extend our results to the case of unconditional product measures with decreasing density.
\end{abstract}
\maketitle

\section{Introduction}
The notion of sumset is a standard operation in the field of additive combinatorics which makes sense for any sets in a commutative group, i.e., for any sets $A,B$ in a commutative group, their sumset  $A+B$, is defined as the set of all possible sums of elements from $A,B$, i.e., $A+B =\{a +b : a \in A, b\in B\}.$ One useful and elegant inequality in additive combinatorics that provides bounds on the cardinality of sumsets of finite sets is the {\it Pl\"unnecke-Ruzsa inequality} \cite{R-70,TV-06}: For any finite sets $A, B_1, \ldots, B_m$ in a commutative group, there exists a set $X \subset A$,
\begin{equation}\label{eq:PR}
    \#(A)^m \#(X + B_1 + \ldots +B_m) \leq \#(X) \prod_{i=1}^m\#(A+B_i).
\end{equation}
Here, $\#(A)$ denotes the cardinality of the set $A$. This inequality was introduced by Pl\"unnecke \cite{P-70} and further improved by Ruzsa \cite{R-97} to the case of compact sets on a locally compact commutative group.

The comparison between inequalities in Information Theory and Convex Geometry have been thoroughly investigated. Bobkov and Madiman \cite{BM-12} introduced an inequality that is inspired by (\ref{eq:PR}): For any convex bodies $A,B,C\subset\R^n$, one has
\[
    \left| A \right| \left| A+B+C \right| \leq 3^n \left| A+B \right| \left| A+C \right|.
\]
The constant $3^n$ is known to not be sharp. Fradelizi, Madiman, and the third named author \cite{FMZ-24} explored this further, determining the lower bound and upper bound of the best possible constant $c_n$,
\begin{equation} \label{c_Plun-Ruz}
    c_n =\sup \frac{\left| A \right| \left| A+B+C \right| }{\left| A+B \right| \left| A+C \right|},
\end{equation}
where the supremum is taken over all convex bodies $A,B,C$ in $\R^n$: 
$$
    (4/3+o(1))^n \leq c_n \leq \varphi^n,
$$ 
where $\varphi = \frac{1+\sqrt5}{2}$ is the golden ratio.

The second named author generalized results from \cite{BM-12} to the case of $m-$convex bodies. More precisely,  let $c_{n,m}$ be the smallest possible constant such that 
\begin{equation} \label{m-Plu-Ruz}
    \left| A \right|^{m-1} \left| A+B_1+\ldots+B_m \right| \leq c_{n,m} \prod_{i=1}^m \left| A+B_i \right|,
\end{equation} 
for any convex bodies $A,B_1, \ldots, B_m$ in $\R^n$. Then,
\[
    \frac{e}{\sqrt{2 \pi n (e-1)}} \left( \frac{4}{3} \right)^n+1/2 \leq c_{n,m} < 2^n.
\]
Moreover, sharper lower and upper bounds depending on $m$ were established in \cite{N-24}.

Ruzsa distance was introduced in \cite{R-96} as $$d(A,B) = \log \frac{\#(A-B)}{\sqrt{\#(A) \#(B)}}, $$ where $A,B$ are finite subsets of a commutative group (see also \cite[Chapter~2.33]{TV-06}). Ruzsa distance is not a metric, however it satisfies the triangle inequality: For any finite sets $A,B,C$ in a commutative group, 
$$
    \#(A) \# (B - C) \leq \#(B - A) \#(A - C ).
$$ 
An analog of this inequality to the case of volume of convex bodies was presented in \cite{FMZ-24}: For any convex bodies $A,B,C$, one has 
\[
    \left| A \right| \left| A+B+C \right| \leq \frac{1}{2^n} \binom{2n}{n}c_n \left| A-B\right| \left| A-C \right|,
\]
where $c_n$ is defined in \eqref{c_Plun-Ruz}. 

We note that the lower bounds for $c_n$ and $c_{n,m}$ in \cite{FMZ-24, N-24} were obtained by considering a convex body which has many symmetries. This inspired us to explore Ruzsa type inequalities for the class of anti-blocking bodies. 

We say that a convex body $K$ is an {\it anti-blocking body} if $K \subset \R^n_+$ and orthogonal projection on every coordinate subspace equals to its section by the same subspace (see Section \ref{B_anti-blocking bodies} for more details). Many crucial inequalities were proved for the class of  anti-blocking bodies see \cite{ASS-23,S-24}. In Corollary \ref{corant}, we prove a sharp analog of \eqref{m-Plu-Ruz}: For any anti-blocking bodies $A,B_1,\ldots, B_m\subset \R^n_+$, one has 
\begin{equation} \label{m-Plu-Ruz-diff}
    \left| A \right|^{m-1} \left| A+B_1+\ldots+B_m \right| \leq \zeta_{n,m} \prod_{i=1}^m \left| A-B_i \right|,
\end{equation}
where $\zeta_{n,m} $ is defined as follows: For any integers $0 \leq \alpha_i \leq n$, let
\begin{equation} \label{s-con_proj}
    d_{n,m}(\alpha_1,\ldots,\alpha_m) :=\frac{\prod_{i=1}^m\binom{\alpha_i}{N}}{\binom{n}{N}^{m-1}} ,
\end{equation}
where $N= \sum_{i=1}^m \alpha_i -n(m-1)\geq 0 $ and 
\begin{equation*} \label{zeta-constant}
    \zeta_{n,m} := \max_{0\leq \alpha_i \leq n} d_{n,m}(\alpha_1,\ldots,\alpha_m).
\end{equation*}


The next observation of this paper is related to the conjecture of the concavity of the ratio of volumes to volumes of projections: For any convex bodies $A,B \subset \R^n$,
\begin{equation} \label{r-vol-proj}
    \frac{\left| A\right|}{\left| P_E A \right|} +  \frac{\left| B\right|}{\left| P_E B \right|} \leq \frac{\left| A + B\right|}{\left| P_E(A + B) \right|},
\end{equation}
where $E$ is a hyperplane of $\R^n$ and $P_E $ is the orthogonal projection onto $E$. The conjecture \eqref{r-vol-proj} is connected to a  question of V. Milman \cite{Schneider},  whether there exists a version of the Bergstrom inequality for convex bodies. It is also connected to Dembo, Cover, Thomas conjecture \cite{DCT-91} which motivated by Fisher information inequality. The conjecture \eqref{r-vol-proj} is known to be true in $\R^2$ (see, for example, \cite{FMMA-24}), but it is false in general starting from dimension $3$, even for the weaker version,
\begin{equation} \label{r-w-vol-proj}
    \frac{\left| A\right|}{\left| P_E A \right|} \leq \frac{\left| A + B\right|}{\left| P_E(A + B) \right|}.
\end{equation}
The counterexample was presented in \cite{FGM03}. It has been demonstrated in \cite{FMMA-24} that \eqref{r-w-vol-proj} is valid for the class of zonoids when $n = 3$ (the case of the zonoids in higher dimensions is still open). We refer to \cite{FMMA-24, FMZ-24}  for detailed discussion of those conjectures and their equivalent versions.

We study an analog of conjecture (\ref{r-vol-proj}) for anti-blocking bodies in the view of Ruzsa's triangle inequality by replacing the Minkowski sum with the difference.  Let $r_{n,i}$ be the best possible constant for the following inequality:
\begin{equation} \label{conj_ratioA}
    \frac{\left| A\right|}{\left| P_E A \right|} +  \frac{\left| B\right|}{\left| P_E B \right|} 
    \leq 
    r_{n,i} \frac{\left| A - B\right|}{\left| P_E(A - B) \right|},
\end{equation}
for any anti-blocking bodies $A,B \subset \R^n_+$ and for any $i-$dimensional coordinate subspace $E$. We prove in Theorem \ref{L1-sum} that  
\begin{equation*} \label{r_constant}
    r_{n,i} = \max_{0 \leq j \leq n} d_{n,2} (i,j),
\end{equation*}
where $d_{n,2} (i,j) $ is defined in \eqref{s-con_proj}.

Inspired by \cite{FMMA-24}, we extend   \eqref{m-Plu-Ruz-diff} and \eqref{conj_ratioA} to a more general operation on the pair of convex bodies introduced by Firey \cite{F-62}, called the {\it $L_p-$sums}.

An upper bound for the volume of difference body $K-K$ is given
by the Rogers-Shephard inequality, originally proven in \cite{RS-57} (see also \cite{Schneider}). The $L_p$ sum analog of this inequality was proved in $\R^2$ by Bianchini and Colesanti \cite{BC-08} and it is an open question to prove the sharp $L_p$-Rogers-Shephard inequality in $\R^n$, $n\ge 3$.  We prove this inequality in the case of locally anti-blocking bodies, where a {\it locally anti-blocking body} can be viewed as a body which is anti-blocking body in each orthant. Let $\kappa_{n,q}$ be the best possible constant for the following inequality:
\begin{equation*}
    |K \oplus_p -K| \leq \kappa_{n,q} |K|,
\end{equation*}
for any locally anti-blocking body $K\subset \R^n$ where $\frac{1}{p} + \frac{1}{q} = 1$ and $p \geq 1$. We prove in Lemma \ref{RS-l-anti} that 
\[
    \kappa_{n,q} = \sum_{i=0}^n \binom{n/q}{i/q}^{-1}  \binom{n}{i}^2,
\]
where the binomial coefficient is defined by the Gamma function: 
    \[
        \binom{x}{y}:= \frac{\Gamma (x+1) }{\Gamma(y+1)\Gamma(x-y+1)}.
    \]

The paper is organized as follows. Section 2 is dedicated to the background materials. In section 3, we present an improvement of Local Loomis-Whitney inequality, of  Bollob\'as and Thomason  \cite{BA-95}, for a $\lambda-$uniform cover. Following this, in section 4, we present the best possible constants of a Pl\"unnecke-Ruzsa type inequality and the concavity of the ratio of volumes to volumes of projections for anti-blocking bodies. Moving on to section 5, we present our results with $L_p-$sums, and also present a sharp constant of a $L_p$-Roger-Shephard inequality for a locally anti-blocking body. Section 6 delves deeper into inequalities when considering volumes with respect to an $s-$concave measure. Finally, Section 7 is an Appendix where we collect optimizations of constants and numerous technical calculations, used in the paper.

\noindent{\bf Acknowledgments.}
We are grateful to Shay Sadovsky for numerous  useful comments  greatly improved the presentation of this paper.

\section{Background}

\subsection{Preliminaries}
We refer to \cite{AGM1,AGM2,Schneider} for basic definitions and facts used in this paper. We write $\langle x,y \rangle $ for {\it the inner product} of $x,y \in \R^n$ and $e_1,\ldots ,e_n$ for the {\it standard orthonormal basis} of $\R^n$. A {\it convex body} is defined as a convex, compact non-empty set. 
If a convex body $K$ containing the origin, then, for any $x \in \mathbb{R}^n$, we define the {\it Minkowski functional} by
$$
\|x\|_K:=\inf \{t>0 : x \in tK\}.
$$
We note that if $0\in \partial K$, then the  above infimum may be taken over the empty set, in this case we set such infimum to be infinity. 
 We also define a  {\it polar body} of $K$ by
\[
    K^\circ := \{y \in \R^n : \langle x, y\rangle \leq 1 \text{ for all } x \in K\}.
\]
For a convex body $K$, we define the {\it support function} of $K$ by 
$
    h_K (x) := \underset{{y\in K}}{\max} \langle x,y \rangle.
$
Note that $ h_K (x) = \| x\|_{K^\circ}$ if $0 \in K$.

We write $ |K| $ for the $m-$dimensional Lebesgue measure (volume) of a measurable set $K$ where $m$ is a dimension of the minimal affine space containing $K$. 
For any sets $K$ and $L$, the sumset $K+L = \{ x+y : x\in K, y \in L \}$ is called {\it Minkowski sum} of $K$ and $L$.



\subsection{Anti-blocking bodies} \label{B_anti-blocking bodies}
A subspace $E$ is a {\it coordinate subspace} if $E$ is spanned by $\{e_i\}_{i\in I}$, where $I \subset [n]$, (where we use the notation $[n]=\{1,2,\dots, n\}$).
Denote by \(\G\) the collection of coordinate subspaces of \(\R^n\) and \(\GK \subset \G\) the collection of coordinate subspace of dimension \(k\).

A convex body \( K \subset \R^n_+\) is called {\it anti-blocking} if for any \( E \in \G \), \(P_E(K) = K \cap E\). Note that the collection of anti-blocking bodies are closed under Minkowski sum.


The following decomposition of anti-blocking bodies was proved in \cite{CFS-17, ASS-23}: for any pair of anti-blocking bodies $A$ and $B$, one has  
\begin{equation}\label{sum_ort}
    A - B = \bigcup_{E\in \G} \left(P_E A  -P_{E^\bot }B \right).
\end{equation}
The above union is almost disjoint with respect to Lebesgue measure which immediately gives a very useful formula: 
\begin{equation} \label{antproj}
    \left| A - B \right|= \sum_{E\in \G} \left| P_E A \right| \left| P_{E^\bot }B \right|.
\end{equation}
It is worth mentioning that we agree that $\left| P_E K \right| = 1$ when $E = \{0\}$. Note that it is clear that \eqref{antproj} holds for any measure that is absolutely continuous with respect to Lebesgue measure.

It was proved in  \cite{ASS-23} that for  any pair of anti-blocking bodies $A$ and $B$, one has
\begin{equation} \label{antproj2}
    \left| A+B \right| \leq   \left| A-B \right|.
\end{equation}
 The proof of (\ref{antproj2}) relies on the Reverse Kleitman inequality \cite{BLR-89}. The inequality (\ref{antproj2}) can be also proved directly using a symmetrization argument, which we reproduce for the case of $L^p-$sums in Lemma \ref{RK-lp}, below. 

Consider  \( \delta \in \{-1,1\}^n \) and  a set \( S \subset \R^n \), we define 
$\delta S$ to be a set such that  \( \delta S := \{ (\delta_1x_1,\ldots,\delta_n x_n) : x \in S \} \). 
A convex body $K$ is {\it locally anti-blocking}, if for any $\delta \in \{-1,1\}^n$, $\delta K \cap \R^n_+$ is anti-blocking. A {\it 1-unconditional  convex body} K is defined by $\delta K=K$, for all \( \delta \in \{-1,1\}^n \). We observe that any 1-unconditional convex body is a locally anti-blocking body. We refer to \cite{ASS-23, S-24} for many more results and inequalities for (locally) anti-blocking bodies. 
\subsection{ Inequalities on the volume of projections} \label{LLW-section}
A classical inequality relating volume of a convex body $K \subset \R^n$ and volume of its projections is the {\it Loomis-Whitney inequality} \cite{LW-49}:
\begin{equation} \label{LW-ori}
    \left| K \right|^{n-1} \leq \prod_{i=1}^n \left| P_{e_i^\perp} K \right|.
\end{equation}
The inequality \eqref{LW-ori} was generalized by Bollob\'as and Thomason in \cite{BA-95}. For every $ \tau \subset [n]$, we define $E_\tau := \Span \{e_i : i\in \tau\}^\perp $. 
For $\sigma \subset [n]$, the sets $\sigma_1,\ldots,\sigma_m \subset \sigma$ form a {\it $\lambda-$uniform cover of $\sigma$}, if every $ j \in \sigma$, it belongs to exactly $\lambda$ of the sets $\sigma_i$. Note that $E_{\sigma}\subset E_{\sigma_i}$ for any $i\in[m]$.

It was proved in \cite{BA-95} that for any compact set $K \subset \R^n$ and for any $\lambda-$uniform cover $(\sigma_1,\ldots,\sigma_m)$ of $[n]$, one has
\begin{equation} \label{BT-int}
    \left| K \right|^{\lambda} \leq \prod_{i=1}^m \left| P_{E_{\sigma_i}^\perp} K \right|.
\end{equation}
The local version of \eqref{BT-int} is a particular case of local versions of Alexandrov-Fenchel's inequality.  The first inequality of such type was proved by Fenchel (see \cite{Fen36} and also \cite{Schneider}) and further generalized in \cite{FGM03, AFO14, SZ16, Xia19}.

A local version of (\ref{BT-int}) was observed by Brazitikos, Giannopoulos and Liakopoulos in \cite{BGL-18}: For any convex body $K \subset \R^n$ and for any $\lambda-$uniform cover $(\sigma_1,\ldots,\sigma_m)$ of $\sigma$, 
\begin{equation} \label{m-proj-BT}
    \left|K\right|^{m-\lambda} \left| P_{E_\sigma} K \right|^\lambda \leq \binom{n}{n-|\sigma|}^{\lambda-m}\binom{n-\frac{\lambda |\sigma|}{m}}{n-|\sigma|}^m \prod_{i=1}^m  \left| P_{E_{\sigma_i}} K \right|,
\end{equation}
where from now we will denote the cardinality of $\sigma$ by $|\sigma|$. 

Inequality (\ref{m-proj-BT}) was further improved by Alonso-Guti\'errez, Artstein-Avidan, Gonz\'alez Merino, Jim\'enez and Villa  in \cite{AAGJV-19}: For any convex body $K\subset \R^n$ and for any 1-uniform cover $(\sigma_1, \sigma_2 )$ of $\sigma$, 
\begin{equation}\label{projcontr}
    |K| |P_{E_\sigma}K|  \leq  \frac{\binom{i}{ i+j-n} \binom{j}{ i+j-n}}{\binom{n}{ i+j-n} }|P_{E_{\sigma_1}} K|\,|P_{E_{\sigma_2}} K|,
\end{equation}
where
  $i = \dim E_{\sigma_1} $, $j =\dim E_{\sigma_2}.$ 


Inequality (\ref{projcontr}) was used in \cite{FMZ-24} to compute the lower bound for the best constant $c_n$ defined in \eqref{c_Plun-Ruz}. 
In Section 3 of this paper, we extend inequality \eqref{projcontr} to the case of $\lambda-$uniform cover $(\sigma_1,\ldots,\sigma_m)$ of $\sigma$. 

Finally, we would like to remind a result of Rogers and Shephard  \cite{RS-58}, who established a sharp lower bound for the volume of a convex body $K \subset\R^n$ in terms of the volumes of its projection and a maximal section:
\begin{equation} \label{lb-RS}
    \left|P_H K\right| \max _{x_0 \in H} \left|K \cap\left(x_0+H^{\perp}\right)\right| \leq\binom{ n}{k} \operatorname{vol}(K),
\end{equation}
where $H$ is a $k-$dimensional subspace of $\R^n$.
\subsection{$s-$concave measure} 
A measure $\mu$ on $\R^n$ is {\it $s-$concave} if for any non-empty compact sets $A,B \subset \R^n$ and for any $t \in [0,1]$, one has 
\begin{equation*}\label{eq:sconc}
    \mu( tA +(1-t)B)^s \geq  t\mu(A)^s + (1-t)\mu(B)^s.
\end{equation*}
The classical Brunn-Minkowski inequality demonstrates that the Lebesgue measure is \(1/n\)-concave (See \cite{Schneider}), i.e.,
\begin{equation} \label{eq:BM}
    | tA +(1-t)B|^{1/n} \geq  t|A|^{1/n} + (1-t)|B|^{1/n}.
\end{equation}
We say that the Borel measure $\mu$ has {\it density}, if it has a locally integrable derivative, i.e.,
$$
    \frac{d \mu(x)}{dx} = \phi(x),
$$
where $\phi:\R^n \rightarrow \R$ and $\phi\in L^1_{loc} (\R^n)$. A measure $\mu$ is {\it unconditional with decreasing density}, if 
\begin{itemize}
    \item[1.] For any $x \in \R^n, \phi(\pm x_1, \ldots,\pm x_n) = \phi(x)$.
    \item[2.] For any $ i \in [n]$ and $x_1,\ldots,x_{i-1},x_{i+1},\ldots, x_n \in \R$, the function $$t \mapsto \phi(x_1,\ldots,x_{i-1},t,x_{i+1},\ldots, x_n)$$ is non-increasing on $[0,\infty)$.
\end{itemize}
Livshyts, Marsiglietti, Nayar and the third named author proved in  \cite{LMNZ-17} that
\begin{theorem} \label{uncon-prod-m}
    Any unconditional, product measure with decreasing density is $1/n$-concave on the class of 1-unconditional convex bodies.
\end{theorem}

\subsection{Berwald inequality}
The classical Berwald inequality was introduced in \cite{B-47, B-73}:
\begin{theorem}[Berwald inequality]
    Let $f $ be a non-negative, concave function supported on a convex body $K\subset \R^n$. Then, for any $-1<p\leq q < \infty$, 
    \begin{equation} \label{Berwald-ine}
        \left( \frac{\binom{n +p}{p}}{| K|} \int_{K} f (y) ^{p} dy \right)^{\frac{1}{p}} 
        \leq 
        \left( \frac{\binom{n+q}{q}}{|K|}\int_{K} f (y) ^{q} dy \right)^{\frac{1}{q}}.
    \end{equation}
\end{theorem}
Recently, Langharst and Putterman \cite{LP-24} extended the inequality \eqref{Berwald-ine} to the case of $s-$concave measures. 
\begin{theorem}[Berwald inequality for $s-$concave measures]
    Fix $s>0$ and let $f $ be a non-negative, concave function supported on a convex body $K\subset \R^n$. Let $\mu$ be a finite, Borel $s-$concave measure supported on $K$.
    Then, for any $-1<p\leq q < \infty$, 
    \begin{equation} \label{Berwald-ine-s}
        \left( \frac{\binom{\frac{1}{s} +p}{p}}{| K|} \int_{K} f (y) ^{p} d\mu(y) \right)^{\frac{1}{p}} 
        \leq 
        \left( \frac{\binom{\frac{1}{s}+q}{q}}{|K|}\int_{K} f (y) ^{q} d\mu(y) \right)^{\frac{1}{q}}.
    \end{equation}
\end{theorem}




\section{Local Loomis-Whitney inequality}

We  remind that  $E_\tau = \Span \{e_i : i\in \tau\}^\perp $, for every $ \tau \subset [n]$ and for $\sigma \subset [n]$, the sets $\sigma_1,\ldots,\sigma_m \subset \sigma$ form a  {\it $\lambda-$uniform cover of $\sigma$}, if every $ j \in \sigma$, it belongs to exactly $\lambda$ of the sets $\sigma_i$.

The Theorem \ref{LM-uni-cov} gives a slight improvement of  \eqref{m-proj-BT}, so that the multiplicative constant in  \eqref{m-proj-BT} depends on the cover and is sharp. The case $|\sigma_1| = \ldots = |\sigma_m|$ of Theorem \ref{LM-uni-cov} is included in \eqref{m-proj-BT}.

\begin{theorem} \label{LM-uni-cov}
    Let $\sigma \subset [n]$, $m \geq \lambda \geq 1$ and let $(\sigma_1, \ldots, \sigma_m)$ be $\lambda-$uniform cover of $\sigma$. For any convex body $K \subset \R^n$, one has
    \begin{equation} \label{L-LM-unif}
        \left| K \right|^{m-\lambda} \left|P_{E_\sigma} K \right|^\lambda 
        \leq 
        \frac{\prod_{i=1}^m \binom{n-|\sigma_i|}{n-|\sigma|} }{\binom{n}{n-|\sigma|}^{m-\lambda}} \prod_{i=1}^m \left| P_{E_{\sigma_i}} K \right|.
    \end{equation}
\end{theorem}

\begin{proof} The proof  is based on the method from \cite{BGL-18}.
    The case $m = \lambda$ is trivial. Assume $m > \lambda$. We have
    \[
         \lambda | \sigma | = \sum_{i=1}^m |\sigma_i|.
    \]
    For every $ y \in P_{E_\sigma} K$, we define the sets
    \[
        K(y) := \{ t \in E_\sigma^\perp : y+t \in K\},
    \]
    and
    \[
        K_i(y) := \{ t \in E_\sigma^\perp\cap E_{\sigma_i} : y+t \in P_{E_{\sigma_i}} K \}.
    \]
    Note that $(\sigma\smallsetminus \sigma_1, \ldots , \sigma \smallsetminus \sigma_m) $ is $(m-\lambda)-$uniform cover of $\sigma$. Using \eqref{BT-int}, we have
    \begin{equation*} \label{LM-p1}
        |K| = \int_{P_{E_\sigma} K} \left| K(y) \right| dy \leq \int_{P_{E_\sigma} K} \prod_{i=1}^m \left| K_i(y) \right|^{\frac{1}{m-\lambda}} dy.
    \end{equation*}
    For each $i$, we define $f_i :P_{E_\sigma} K \rightarrow [0,\infty)$ by
    \[
        f_i (y) := \left| K_i(y) \right|^{\frac{1}{| \sigma |-|\sigma_i|}}.
    \]
    Using H\"{o}lder inequality, we have
    \[
        |K|
        \leq 
        \int_{P_{E_\sigma} K} \prod_{i=1}^m f_i(y)^{\frac{| \sigma |-|\sigma_i|}{m-\lambda}} dy 
        \leq 
        \prod_{i=1}^m \left(  \int_{P_{E_\sigma} K} f_i (y) ^{| \sigma |} dy \right)^{\frac{| \sigma |-|\sigma_i|}{| \sigma |(m-\lambda)}}.
    \]
    Applying (\ref{eq:BM}), we get that $f_i$ is concave, thus, we can apply the Berwald inequality \eqref{Berwald-ine} to get
    \[
        \left( \frac{\binom{n}{n-| \sigma |}}{|P_{E_\sigma} K|} \int_{P_{E_\sigma} K} f_i (y) ^{| \sigma |} dy \right)^{\frac{1}{| \sigma |}} 
        \leq 
        \left( \frac{\binom{n-|\sigma_i|}{n-| \sigma |}}{|P_{E_\sigma} K|}\int_{P_{E_\sigma} K} f_i (y) ^{| \sigma |-|\sigma_i|} dy \right)^{\frac{1}{| \sigma |-|\sigma_i|}}.
    \]
    Hence,
    \begin{align*}
        |K|^{m-\lambda} 
        &\leq \prod_{i=1}^m \left(  \int_{P_{E_\sigma} K} f_i (y) ^{| \sigma |} dy \right)^{\frac{| \sigma |-|\sigma_i|}{| \sigma |}}
        \\
        &\leq
        \frac{1}{\binom{n}{n-| \sigma |}^{m-\lambda}|P_{E_\sigma} K|^\lambda} \prod_{i=1}^m \binom{n-|\sigma_i|}{n-| \sigma |} \int_{P_{E_\sigma} K} f_i (y) ^{| \sigma |-|\sigma_i|} dy 
        \\
        &=
        \frac{1}{\binom{n}{n-| \sigma |}^{m-\lambda}|P_{E_\sigma} K|^\lambda} \prod_{i=1}^m \binom{n-|\sigma_i|}{n-| \sigma |} \left| P_{E_{\sigma_i}} K \right|. \qedhere
    \end{align*}
\end{proof}
\begin{remark} \label{HannerP}
    The equality in  \eqref{L-LM-unif} can be achieved by using the Hanner polytope,
    \[
        K = \Conv\left\{ \sum_{i\in \sigma} [-e_i,e_i], \sum_{i \not\in \sigma}[-e_i,e_i] \right\}.
    \]
    We include a detailed computation  in the Appendix, Lemma \ref{Cal-LM}.
\end{remark}

\begin{corollary} \label{1-uni-prof}
    Let $\sigma \subset [n]$, $m \geq 1$ and $(\sigma_1, \ldots, \sigma_m)$ be $1-$uniform cover of $\sigma$. For any convex body $K \subset \R^n$, we have
    \begin{equation} \label{m-proj}
        \left| K \right|^{m-1} \left|P_{E_\sigma} K \right| \leq d_{n,m} (\alpha_1,\ldots,\alpha_m) \prod_{i=1}^m \left| P_{E_{\sigma_i}} K \right|,
    \end{equation}
    where $ \alpha_i = \dim E_{\sigma_i}$, and 
\begin{equation}\label{eq:dnm}
        d_{n,m}(\alpha_1,\ldots,\alpha_m) := \frac{\prod_{i=1}^m \binom{\alpha_i}{N} }{\binom{n}{N}^{m-1}},
\end{equation} 
    where $N = \dim E_\sigma = \sum_{i=1}^m \alpha_i -n(m-1) \geq 0.$ 
\end{corollary}
\begin{remark}
We note that    \eqref{projcontr} is a particular case of \eqref{m-proj} when $m =2$.
\end{remark}
\section{Inequalities for Anti-blocking bodies}

We start with a lemma which provides an upper bound on the Minkowski sum of anti-blocking bodies using the volume of their  projections.
\begin{lemma}
    Fix a positive integer $r$ and  let $A_1,\ldots,A_r$ be anti-blocking bodies in $\R^n$. Then,
    \begin{equation} \label{upbd-anti-blocking}
        \left| A_1 + \ldots + A_r \right| 
        \leq 
        \sum_{(\sigma_1,\ldots,\sigma_r)} \left( \prod_{i=1}^r \left| P_{E_{\sigma_{i}}^\perp} A_i\right| \right),
    \end{equation}
    where the sum is taken over all 1-uniform covers $(\sigma_1,\ldots,\sigma_r)$ of $[n]$.
\end{lemma}

\begin{proof}
    The proof is by induction on $r$. The case $ r= 1$ is trivial. Using inductive assumption, we get 
    \begin{equation*} \label{upbd-anti-blocking-fp1}
        \left| A_1 + \ldots + A_r \right| 
        \leq 
        \sum_{(\sigma_1,\ldots,\sigma_{r-1})} \left( \left| P_{E_{\sigma_{r-1}}^\perp} (A_{r-1}+A_r)\right|  \cdot \prod_{i=1}^{r-2} \left| P_{E_{\sigma_{i}}^\perp} A_i\right| \right).
    \end{equation*}
    From \eqref{antproj2} and \eqref{antproj}, we have
    \begin{align*}
        \left| P_{E_{\sigma_{r-1}}^\perp} (A_{r-1}+A_r)\right| &\leq \left| P_{E_{\sigma_{r-1}}^\perp} A_{r-1}- P_{E_{\sigma_{r-1}}^\perp}A_r\right| 
        \\
        &= \sum_{(\sigma_{r-1}',\sigma_r')} \left| P_{E_{\sigma_{r-1}'}^\perp} A_{r-1} \right| \left| P_{E_{\sigma_{r}'}^\perp} A_r \right|,
    \end{align*}
    where the sum is taken over 1-uniform cover $ (\sigma_{r-1}',\sigma_r')$ of $\sigma_{r-1}$. Then,
    \begin{equation*} 
        \left| A_1 + \ldots + A_r \right| 
        \leq 
        \sum_{(\sigma_1,\ldots,\sigma_{r-2},\sigma_{r-1}',\sigma_r')} \left( 
        \left| P_{E_{\sigma_{r-1}'}^\perp} A_{r-1} \right| \left| P_{E_{\sigma_{r}'}^\perp} A_r \right|
         \prod_{i=1}^{r-2} \left| P_{E_{\sigma_{i}}^\perp} A_i\right| \right).
    \end{equation*}
    The proof is finished since $(\sigma_1,\ldots,\sigma_{r-2},\sigma_{r-1}',\sigma_r')$ is 1-uniform of $[n]$. 
\end{proof}

\begin{theorem} \label{m-upcn-sum}
    Fix $n,m\in \mathbb{N}$ and let $c_1(n,m)$ be the best constant such that for any anti-blocking bodies $A,B_1,\ldots,B_m $ in $\R^n$, one has
    \begin{equation*} \label{bztm}
        \left| A\right|^{m-1} \left| A-B_1 - \cdots -B_m \right| \leq c_1(n,m) \prod_{i=1}^m\left| A - B_i \right|.
    \end{equation*} 
    Let $c_2(n,m)$ be the best constant such that for any anti-blocking body $A$ in $\R^n$ and for any 1-uniform cover $(\sigma_1,\ldots,\sigma_m)$ of $\sigma$, one has 
            \begin{equation} \label{m-upcn}
                 |A|^{m-1}|P_{E_\sigma}A|  \leq  c_2(n,m) \prod_{i=1}^m |P_{E_{\sigma_i}} A|.
            \end{equation}
    Then, $c_1(n,m) = c_2(n,m) $.
\end{theorem}

\begin{proof}
The existence of constants $c_1(n,m)$ and $c_2(n,m)$ are known (see, for example, \cite{N-24}), moreover,   it was proved in \cite[Theorem~10]{N-24} that $c_2(n,m) \leq c_1(n,m)$. Using \eqref{antproj} and \eqref{upbd-anti-blocking}, we have
    \begin{align*}
         \left| A\right|^{m-1} &\left| A-B_1 - \cdots -B_m \right|
         \\
         &\leq
          \left| A\right|^{m-1} \sum_{\sigma \subset [n]} \left| P_{E_{\sigma}} A \right| \left| P_{E_\sigma^\perp} (B_1 + \cdots +B_m) \right|
          \\
          &\leq 
          \left| A\right|^{m-1} \sum_{\sigma \subset [n]} \left(\left| P_{E_{\sigma}} A \right| \sum_{(\sigma_1,\ldots,\sigma_m)} \left(\prod_{i=1}^m \left| P_{E_{\sigma_{i}}^\perp} B_i\right| \right) \right),
    \end{align*}
    where $(\sigma_1,\ldots,\sigma_r)$ is 1-uniform cover of $\sigma$.
    Using \eqref{antproj}, we have 
    \begin{equation} \label{Cal-T9}
        \prod_{i=1}^m\left| A - B_i \right| = \prod_{i=1}^m \left( \sum_{G_i \in \G} \left(\left| P_{G_i} A \right| \left| P_{G_i^\perp} B_i \right| \right)\right).
    \end{equation}
    Using \eqref{m-upcn} and \eqref{Cal-T9}, we have
    \begin{align*}
        \left| A\right|^{m-1} &\left| A-B_1 - B_2- \cdots -B_m \right|
        \\
        &\leq 
        \sum_{\sigma \subset [n]} \left(c_2(n,m) \prod_{i=1}^m |P_{E_{\sigma_i}} A| \sum_{(\sigma_1,\ldots,\sigma_m)} \prod_{i=1}^m \left| P_{E_{\sigma_{i}}^\perp} B_i\right| \right)
        \\
        &\leq \sum_{\sigma \subset [n]} \left(c_2(n,m) \sum_{(\sigma_1,\ldots,\sigma_m)} \prod_{i=1}^m \left(|P_{E_{\sigma_i}} A|  \left| P_{E_{\sigma_{i}}^\perp} B_i\right| \right) \right)
        \\
        &\leq c_2(n,m) \prod_{i=1}^m \left( \sum_{G_i \in \G} \left(\left| P_{G_i} A \right| \left| P_{G_i^\perp} B_i \right| \right)\right)
        \\
        &= c_2(n,m) \prod_{i=1}^m\left| A - B_i \right|.
    \end{align*}
    This means that $c_1(n,m) \leq c_2(n,m)$. 
\end{proof}

The following Corollary \ref{corant} is a consequent of the proof of Theorem \ref{m-upcn-sum} and Corollary \ref{1-uni-prof}. Moreover, together with \eqref{antproj2}, it implies inequality \eqref{m-Plu-Ruz-diff} with the best possible constant $\zeta_{n,m}$ for the case of anti-blocking bodies. 

\begin{corollary}\label{corant}
    For anti-blocking bodies $A,B_1,\ldots,B_m \subset \R^n_+,$ we have 
    \begin{equation} \label{m-upcn-sum-wc}
        \left| A\right|^{m-1} \left| A\pm B_1 \pm \cdots \pm B_m \right| 
        \leq 
        \zeta_{n,m} \prod_{i=1}^m\left| A - B_i \right|,
    \end{equation}
    where 
    $$
        \zeta_{n,m} := \underset{0\leq \alpha_i \leq n}{\max} d_{n,m}(\alpha_1,\ldots,\alpha_m),
    $$
    where $d_{n,m}$ is defined in (\ref{eq:dnm}).  Moreover, $ \zeta_{n,m} $ is the best possible constant. 
\end{corollary}

\begin{remark}
    We provide the explicit formula of the constant $ \zeta_{n,m} $ when $ m = 2$ in Lemma \ref{cn-calculation} (see Appendix below), 
    \[
       \zeta_{n,2} =
       \begin{cases}
           \displaystyle \frac{(2k)!^3}{(3k)!k!^3}; & n =3k, \\
           \displaystyle \frac{(2k+1)^2}{(k+1)(3k+1)}\frac{(2k)!^3}{(3k)!k!^3}; & n=3k+1, \\
           \displaystyle \frac{2(2k+1)^3}{(k+1)(3k+1)(3k+2)}\frac{(2k)!^3}{(3k)!k!^3}; & n =3k+2.
       \end{cases}
   \]
\end{remark}

\begin{theorem}\label{L1-sum}
    Let \( A,B \subset \R^n_+\) be anti-blocking bodies. Then, for any proper coordinate subspace \( E \in \GK[n][i] \),
    \begin{equation} \label{ratiofp}
        \frac{\left| A\right|}{\left| P_E A \right|} +  \frac{\left| B\right|}{\left| P_E B \right|} 
        \leq 
        r_{n,i}  \frac{\left| A - B\right|}{\left| P_E(A - B) \right|} ,
    \end{equation}
    where 
    \[
        r_{n,i} := \max_{0\leq j\leq n} d_{n,2}(i,j),
    \]
    where $d_{n,2}$ is defined in (\ref{eq:dnm}).
    Moreover, the constant \(r_{n,i}\) is the best possible constant.
\end{theorem}
We will prove Theorem  \ref{L1-sum} as a part (case $p=1$) of a more general result on Firey sums (see  Theorem \ref{ratiolp-thm} below).

\begin{remark}
    We provide the explicit formula of the constant $r_{n,i}$ in Proposition \ref{rn-calculation},
    \[
        r_{n,i} =
           \begin{cases}
               \displaystyle \frac{(2k)!(n-k)!^2}{k!^2(n-2k)!n!}; & i =2k, \\
               \displaystyle \frac{(2k+1)(n-2k)}{(k+1)(n-k)} \frac{(2k)!(n-k)!^2}{k!^2(n-2k)!n!}; & i=2k+1.
           \end{cases}
    \]
\end{remark}

\section{$L_p-$sums of anti-blocking bodies}
Firey \cite{F-62} extended the concept of Minkowski sum to $L_p-$sums of convex bodies $K$ and $L$ containing the origin, so-called $K \oplus_p L$ where $p \geq 1$, by using support functions:
\[
    h_{K\oplus_p L}^p=h_K^p +h_L^p. 
\]
The theory of Firey sum was futher developed by Lutwak, who created the notion of $L_p-$mixed volumes and proved a number of very useful inequalities \cite{L-93,L-96}.
Lutwak, Yang, Zhang \cite{LYZ-12} extended the notion of $L_p-$sums to the case of non-convex sets: for any  $K,L \subset \R^n$, $p>1$:
$$
    K\oplus_p L:= \{ (1-t)^{1/q}x +t^{1/q}y :x\in K, y\in L ,t \in [0,1]\},
$$
where $ \frac{1}{q} + \frac{1}{p} = 1 $. 
The following lemma is a well-known formula for the volume of a direct Firey sum.
\begin{lemma}
Let $E$ be an $i-$dimensional subspace of $\R^n$ and $A,B$  be convex bodies, containing the origin,  such that $A \subset E$ and $B \subset E^\bot$. Then, for any $p\geq 1$,
    \begin{equation}\label{vol_lp_sum}
        |A\oplus_p B|=
        \binom{n/q}{i/q}^{-1} 
        \left|A\right|\left|B\right|.
    \end{equation}
\end{lemma}

\begin{proof}
    Using that $A$ and $B$ belong to orthogonal subspaces, we get that
    \[
        \|x\|_{A\oplus_p B}^q= \|x\|^q_{(A^\circ \oplus_q B^\circ)^\circ} = h^q_{A^\circ \oplus_q B^\circ} = h^q_{A^\circ } +h^q_{B^\circ } = \|x_1\|_{A}^q +\|x_2\|_{B}^q,
    \]
    for any $x\in \R^n$ and  $x_1 = P_E x,x_2 = P_{E^\perp} x$.
    Next, using that
    \begin{equation}\label{eq:volint}
    | K| =\Gamma\left( 1+\frac{n}{q}\right)^{-1} \int_0^\infty e^{-\|x\|_{K}^q} dx,
    \end{equation}
    for a convex body $K \subset \R^n$, containing the origin (see, for example, \cite{RZ}), we get
    \begin{align*}
        \left| A \oplus_p B \right| &= \frac{1}{\Gamma\left( 1+\frac{n}{q}\right)} \int_0^\infty e^{-\|x\|_{A\oplus_p B}^q} dx
        \\
        &= \frac{1}{\Gamma\left( 1+\frac{n}{q}\right)} \int_{E}  e^{-\|x_1\|_{A}^q} dx_1  \int_{E^\perp} e^{-\|x_2\|_{B}^q} dx_2
        \\
        &= 
        \frac{\Gamma(1+\frac{i}{q})\Gamma\left(1+\frac{n-i}{q}\right)}{\Gamma\left(1+\frac{n}{q}\right)}
        |A|| B|. \qedhere
    \end{align*}
\end{proof}

\begin{lemma}
    Let $A,B \subset \R^n_+$ be anti-blocking bodies. Then, for any $p\geq 1$,
    \begin{equation} \label{vol_lp_sum_2}
        |A\oplus_p - B|= \sum_{i=0}^n \sum_{E\in \GK[n][i]}  
        \binom{n/q}{i/q}^{-1} 
        |P_E A||P_{E^\bot} B|.
    \end{equation}
\end{lemma}
\begin{proof}
    \begin{align*}
        A\oplus_p -B &= \bigcup_{t\in [0,1]} \left\{ (1-t)^{1/q}x+ t^{1/q}y : x\in A,y\in -B\right\}
        \\
        &=\bigcup_{t\in [0,1]} \left((1-t)^{1/q}A- t^{1/q}B\right).
    \end{align*}
    Therefore, from  \eqref{sum_ort}, 
    \begin{align*} 
    A\oplus_p -B &= \bigcup_{t\in [0,1]} \left((1-t)^{1/q}A- t^{1/q}B\right)\\
    &=\bigcup_{t\in [0,1]}  \bigcup_{E\in \G}\left( (1-t)^{1/q}P_E A - t^{1/q}P_{E^\bot}B\right)
    \\
    &=\bigcup_{E\in \G}\left( P_E A\oplus_p - P_{E^\bot} B\right).
    \end{align*}
    It follows that
    $$
        |A\oplus_p -B|= \sum_{i=0}^n\sum_{E\in \GK[n][i]}  |P_E A\oplus_p -P_{E^\bot}B|.
    $$
    We use  \eqref{vol_lp_sum}  to finish the proof.
\end{proof}

\begin{theorem} \label{ratiolp-thm}
    Let $p\geq 1$, $i\in [n-1]$ and $\nu(n,p,i)$ be the smallest constant such that for any anti-blocking bodies $A,B \subset \R^n_+$ and for any  coordinate subspace $E \in \GK[n][i]$ we have
    \begin{equation} \label{ratiolp}
      \frac{|A|}{|P_E A|} + \frac{|B|}{|P_E B|} \leq \nu(n,p,i) \frac{|A\oplus_p -B|}{|P_E (A\oplus_p -B)|}.
    \end{equation}
    Then, 
    $$
        \nu(n,p,i) 
        = 
        \underset{0\leq j \leq n}{\max }\, \left(
        d_{n,2} (i,j) \frac{\Gamma\left(1+\frac{i+j-n}{q}\right)\Gamma\left(1+\frac{n}{q}\right)}{\Gamma\left(1+\frac{i}{q}\right)\Gamma\left(1+\frac{j}{q}\right)}\right),
    $$
    where $d_{n,2}$ is defined in (\ref{eq:dnm}). 
\end{theorem}

\begin{proof}
     Let 
    \[
        \beta := \underset{0\leq j \leq n}{\max }\, \left(
        d_{n,2} (i,j) \frac{\Gamma\left(1+\frac{i+j-n}{q}\right)\Gamma\left(1+\frac{n}{q}\right)}{\Gamma\left(1+\frac{i}{q}\right)\Gamma\left(1+\frac{j}{q}\right)}\right).
    \]
    To prove that $\nu(n,p,i) \leq \beta$, it's enough to show that
    \begin{equation*}            
            \left( \left|A\right|\left|P_E B\right| \right.+\left. \left|B\right|\left|P_E A\right|\right)\left|P_E (A\oplus_p -B)\right| 
            \\
            \leq d\left|P_E A\right|\left|P_E B\right|\left|A\oplus_p -B\right|.
    \end{equation*}
    Using \eqref{vol_lp_sum}, one has
    \begin{equation} \label{L1}
        \begin{split}
            |A|&|P_E B||P_E \left(A\oplus_p -B\right)|
            \\
            = &\sum_{k=0}^i \underset{F \subset E}{\underset{F\in \GK[n][k],}{\sum}} \left(\frac{\Gamma(1+\frac{k}{q})\Gamma\left(1+\frac{i-k}{q}\right)}{\Gamma\left(1+\frac{i}{q}\right)}|A||P_E B||P_F A||P_{F^\bot\cap E} B| \right),
        \end{split}
    \end{equation}
    also 
    \begin{equation}\label{L2}
        \begin{split}
            |B|&|P_E A||P_E (A\oplus_p -B)|
            \\ 
             = &\sum_{k=0}^i \underset{G \subset E}{\sum_{G\in \GK[n][k],}} \left( \frac{\Gamma(1+\frac{k}{q})\Gamma\left(1+\frac{i-k}{q}\right)}{\Gamma\left(1+\frac{i}{q}\right)}|B||P_E A||P_G A||P_{G^\bot \cap E} B|\right),
        \end{split}
    \end{equation}
    and
    \begin{equation}\label{L3}
        \begin{split}
            |P_E A|&|P_E B||A\oplus_p -B|
            \\
            = & \sum_{j=0}^n \sum_{H\in \GK[n][j]}  \left(\frac{\Gamma(1+\frac{j}{q})\Gamma\left(1+\frac{n-j}{q}\right)}{\Gamma\left(1+\frac{n}{q}\right)}|P_E A||P_E B||P_H A||P_{H^\bot} B| \right).
        \end{split}
   \end{equation}
Next,  we will bound each term of \eqref{L1} and \eqref{L2} by corresponding terms of \eqref{L3}. First fixing subspace $F$ in \eqref{L1} and taking  $H^\bot=F^\bot\cap E$ in \eqref{L3}, it is enough to prove that 
    \begin{equation}\label{eq:com1}
        \frac{\Gamma\left(1+\frac{i+j-n}{q}\right)}{\Gamma\left(1+\frac{i}{q}\right)}|A||P_F A| \leq \beta \frac{\Gamma\left(1+\frac{j}{q}\right)}{\Gamma\left(1+\frac{n}{q}\right)}|P_E A||P_H A|.
    \end{equation} 
    Next, repeating the same procedure for \eqref{L2} and \eqref{L3} with $H=G$, it is enough to show that
   \begin{equation}\label{eq:com2}       \frac{\Gamma\left(1+\frac{i-j}{q}\right)}{\Gamma\left(1+\frac{i}{q}\right)}|B||P_{G^\bot\cap E} B| \leq \beta \frac{\Gamma\left(1+\frac{n-j}{q}\right)}{\Gamma\left(1+\frac{n}{q}\right)}|P_E B||P_{H^\bot} B|.
   \end{equation} 
    We remark that we used terms of \eqref{L3} at most once.
    Using \eqref{projcontr}, we see that it is sufficient to take 
    $$
        \nu(n,p,i) \leq \beta.
    $$ 
    to satisfy inequalities (\ref{eq:com1}) and (\ref{eq:com2}).
    
 Let us now show that   $\beta$ is the best possible constant in the class of anti-blocking bodies. 
    Fix $j$, let \( H \in \GK[n][j]\) be such that \(H^\perp \subset E \).
    Let's observe when \(B = U + \epsilon V\), with \( U \) is a unit cube in \(H^\perp\) and \(V\) is a unit cube in \( H \). Let \(A = tK\), where \( t>0 \) and $K$ is an intersection of the Hanner polytope in Remark \ref{HannerP} and the positive orthant. Note that $K$ is an anti-blocking body satisfying  the equality case of inequality \eqref{m-proj}, i.e.,
    \begin{equation} \label{m-proj-eq}
        \frac{\left| K\right|\left| P_{E\cap H} K \right| }{\left| P_E K \right|\left| P_H K\right|} = d_{n,2}(i,j).
    \end{equation}
    Then, by equation \eqref{vol_lp_sum_2},
    \begin{align*}
        &\left| A\oplus_p -B \right| 
        \\
        &= \sum_{m=0}^n \sum_{G\in \GK[n][m]} \left( \frac{\Gamma(1+\frac{m}{q})\Gamma\left(1+\frac{n-m}{q}\right)}{\Gamma\left(1+\frac{n}{q}\right)}t^{m}|P_G K||P_{G^\perp} (U+\epsilon V)| \right).
    \end{align*}
    Note that
    \[
        |P_{G^\perp} U |_{n-m}= 
        \begin{cases}
            1 &\text{; if } H \subset G
            \\
            0 &\text{; otherwise }
        \end{cases},
    \]
    where by $|L|_i$ we denote the $i-$dimensional Lebesgue volume of $L$.
    If we take $ \epsilon = 0$, then the right hand side will be
    \[
        \beta_1 :=\sum_{m=j}^n \underset{H\subset G}{\sum_{G\in \GK[n][m], }} \left( \frac{\Gamma(1+\frac{m}{q})\Gamma\left(1+\frac{n-m}{q}\right)}{\Gamma\left(1+\frac{n}{q}\right)}t^{m}|P_G K| \right).
    \]
    Using \eqref{vol_lp_sum_2}, we have
    \begin{align*}
        &\left| P_E  (A\oplus_p -B) \right|
        \\
        & = \sum_{m=0}^i \underset{F \subset E}{\sum_{F\in \GK[n][m],}} \left( \frac{\Gamma(1+\frac{m}{q})\Gamma\left(1+\frac{i-m}{q}\right)}{\Gamma\left(1+\frac{i}{q}\right)} t^m|P_F K||P_{F^\bot\cap E} (U+\epsilon V)| \right).
    \end{align*}
    Taking $ \epsilon  = 0$, the terms in the right hand side that will remain in the sum are those where $E\cap H \subset F$, i.e.,
    \[
        \beta_2 := \sum_{m=i+j-n}^i \underset{E\cap H \subset F \subset E}{\underset{F\in \GK[n][m],}{\sum}} \left( \frac{\Gamma(1+\frac{m}{q})\Gamma\left(1+\frac{i-m}{q}\right)}{\Gamma\left(1+\frac{i}{q}\right)} t^m|P_F K| \right).
    \]
    By taking $\epsilon =0$ in \eqref{ratiolp}, we have
    \begin{align*}
        t^{n-i}\left(\frac{\left| K\right|}{\left| P_E K \right|} \right)
        &\leq \nu(n,p,i) \frac{\beta_1}{\beta_2}.
    \end{align*}
    Dividing the above inequality by \(t^{n-i} \) and taking \( t= 0\), we get equation
    \[
        \frac{\left| K\right|}{\left| P_E K \right|}  
        \leq \nu(n,p,i) 
        \frac{\Gamma(1+\frac{j}{q})\Gamma(1+\frac{i}{q})\left| P_H K\right| }{\Gamma(1+\frac{i+j-n}{q})\Gamma(1+\frac{n}{q})\left| P_{E\cap H} K \right|}.
    \]
    Using \eqref{m-proj-eq}, we have 
    \[
        d_{n,2}(i,j) \frac{\Gamma(1+\frac{i+j-n}{q})\Gamma(1+\frac{n}{q})}{\Gamma(1+\frac{j}{q})\Gamma(1+\frac{i}{q})} \leq \nu(n,p,i).
    \]
    The proof is completed by taking the supremum over $j$.
\end{proof}
The next corollary follows directly from  Theorem \ref{ratiolp-thm}.
\begin{corollary}
    Let $p\geq 1$ and $i \in [n-1]$. For any anti-blocking bodies $A,B \subset \R^n_+$, then, for any coordinate subspace $E \in \GK[n][i]$, we have
\[
\left(\frac{|A|}{|P_E A|}\right)^p + \left(\frac{|B|}{|P_E B|}\right)^p \leq \nu^p(n,p,i) \left(\frac{|A\oplus_p -B|}{|P_E (A\oplus_p -B)|}\right)^p.
\]

\end{corollary}
We borrow the idea from \cite{BLR-89} to prove  a generalization of \eqref{antproj2} to the case of $L_p-$sums.
Let $K$ be a convex body. For each coordinate hyperplane $e_i^\perp $, we define a symmetrization of $K$, $S_i(K)$ as follows: For each straight line $L$ orthogonal to $e_i^\perp$ such that $L \cap K \neq \emptyset $, shift the line segment $K \cap L$ along $L$ until the segment is in the positive half-space of hyperplane, with one of its endpoints on the hyperplane $e_i^\perp$.
More precisely,
\[
     S_i(K) = \left\{ x + ye_i : x\in P_{e_i^\perp} K, 0 \leq y \leq \left| K \cap (x + \R e_i) \right| \right\}.
\]
Note that, 
\begin{itemize}
    \item $S_i$ preserves volume, i.e. $ \left| S_i(K) \right| = \left| K \right|$.
    \item For $\lambda >0$, $S_i(\lambda A) = \lambda S_i (A)$.
    \item $S_{i} (K) \subset S_i (K')$ if $K \subset K'$.
\end{itemize}
The case of anti-blocking body, $S_i(K) = K$. 
\begin{lemma} \label{Sym-lp}
    Let $A,B $ be convex bodies. Then, $$S_i(A) \oplus_p S_i(B) \subset S_i(A \oplus_p B).$$
\end{lemma}

\begin{proof}
    The proof follows from the definition of $L_p-$sums and the properties of the symmetrization:
    \begin{align*}
        S_i(A) \oplus_p S_i(B) &= \bigcup_{t\in [0,1]} \left(t^{1/q}S_i(A) + (1-t)^{1/q} S_i(B) \right)
        \\
        &= \bigcup_{t\in [0,1]} \left(S_i(t^{1/q}A) +  S_i((1-t)^{1/q}B)\right)
        \\
        &\subset S_i\left(\bigcup_{t\in [0,1]} \left(t^{1/q}A + (1-t)^{1/q}B\right)\right)
        \\
        &= S_i(A\oplus_p B). \qedhere
    \end{align*}
\end{proof} 

\begin{lemma} \label{RK-lp}
    Let $A,B \subset \R^n_+$ be anti-blocking bodies. Then,
    \begin{equation} \label{diff_lp_vol}
        \left| A \oplus_p B \right| \leq \left| A \oplus_p -B \right|.
    \end{equation}
\end{lemma}

\begin{proof}
    Using Lemma \ref{Sym-lp}, we have
    \[
        \left| A \oplus_p -B \right| = \left| S_n(A \oplus_p -B) \right| \geq \left| S_n(A) \oplus_p S_n(-B) \right| = \left| A \oplus_p S_n(-B) \right|.
    \]
    We repeat the process by applying $S_1, \ldots, S_{n-1}$ to get
    \[
        \left| A \oplus_p -B \right| \geq \left|A \oplus_p S_1 S_2 \cdots S_n (-B) \right|.
    \]
    We finish the proof using that $S_1 S_2 \cdots S_n (-B) = B$.
\end{proof} 

\begin{theorem} \label{m-plun-ruz-lp}
    Let $p \geq 1$ and $b(n,p)$ be the smallest constant such that for any anti-blocking bodies $A,B,C \subset \R^n_+$, the following statement holds:
    \begin{equation} \label{Plun-Ruz-lp}
        \left|A \right| \left|A\oplus_p -B \oplus_p -C \right| \leq b(n,p) \left|A\oplus_p -B\right|\left|A\oplus_p -C\right|
    \end{equation}
    Then,
    \[
        b(n,p) = \underset{0\leq i,j \leq n}{\max }\,
        \left(d_{n,2} (i,j) \frac{\Gamma\left(1+\frac{i+j-n}{q}\right)\Gamma\left(1+\frac{n}{q}\right)}{\Gamma\left(1+\frac{i}{q}\right)\Gamma\left(1+\frac{j}{q}\right)}\right),
    \]
    where $d_{n,2}(i,j)$ is defined in (\ref{eq:dnm}). 
\end{theorem}
\begin{proof}
    We follow the proof of Theorem \ref{m-upcn-sum}.
    Let 
    \[
        \beta := \underset{0\leq i,j \leq n}{\max }\,
        \left(d_{n,2} (i,j) \frac{\Gamma\left(1+\frac{i+j-n}{q}\right)\Gamma\left(1+\frac{n}{q}\right)}{\Gamma\left(1+\frac{i}{q}\right)\Gamma\left(1+\frac{j}{q}\right)}\right),
    \]
    To prove that $b(n,p) \leq \beta$, it is enough to check that 
    \[
        |A ||A\oplus_p -B \oplus_p -C| \leq \beta |A\oplus_p -B||A\oplus_p -C|.
    \]
    Using \eqref{vol_lp_sum_2} and \eqref{diff_lp_vol}, we have
    \begin{align*}
        &\left| A\right|\left| A\oplus_p -B\oplus_p -C \right| 
        \\
        &= \left| A \right| \sum_{k=0}^n \sum_{E\in \GK[n][k]} \left( \frac{\GM[k] \GM[n-k]}{\GM[n]} \left|P_EA\right| \left|P_{E^\perp} (B\oplus_p C)\right| \right)
        \\
        &\leq \left| A\right|\sum_{k=0}^n \sum_{E\in \GK[n][k]}  \left( \frac{\GM[k] \GM[n-k]}{\GM[n]}\left|P_EA\right|\left| P_{E^\perp} B \oplus_p - P_{E^\perp}C\right| \right)
        \\
        &= \left| A \right| \sum_{k=0}^n \sum_{E\in \GK[n][k]} \sum_{l=0}^{n-k} \underset{ H \subset E^\perp }{\sum_{ H\in \GK[n][l],}}
        \left( \beta_1(k,l)\left|P_EA\right|\left| P_{H} B \right|\left| P_{(E+H)^\perp}C\right| \right),
    \end{align*}
    where 
    $$
        \beta_1(k,l) := \frac{\GM[k]  \GM[l] \GM[n-k-l]}{\GM[n] }.
    $$ 
    Using \eqref{vol_lp_sum_2}, we have
    \begin{equation} \label{sp-lp}
        \begin{split}
        &\left| A\oplus_p -B \right| \left| A \oplus_p -C \right| 
        \\
        &= 
        \sum_{i=0}^n \sum_{F\in \GK[n][i]}\sum_{j=0}^n \sum_{G\in \GK[n][j]}
        \left( \beta_2(i,j)\left| P_F A \right| \left| P_G A\right|\left| P_{F^\perp} B  \right| \left| P_{G^\perp} C \right| \right),
        \end{split}
    \end{equation}
    where
    \[
        \beta_2 (i,j):=\frac{\GM[i]\GM[n-i]\GM[j]\GM[n-j]}{\GM[n]^2}.
    \]
    It is enough to consider only the terms with \( F= H^\perp \) and \( G = E+H\) in the right hand side of \eqref{sp-lp}. By comparing term by term, it's enough to check that
    \begin{align*}
        &\beta_1 (i+j-n, n-i)\left| A \right| \left|P_{F \cap G}A\right|\leq \beta \cdot \beta_2 (i,j)\left| P_F A \right|\left| P_G A\right|,
    \end{align*}
    where \( F+G = \R^n\) and \( F \cap G = E\) which is true by the definition of $d$.

    Next, we will prove that $\beta$ is the best constant. 
    Fix $i,j$ such that $i+j-n \geq 0$, let $E\in\GK[n][i],F \in \GK[n][j]$ be such that $F+E = \R^n$. Let $ B =U,  C= V$ be unit cubes in $E^\perp $ and $ F^\perp$, respectively. Let $A = tK$ for $t >0$ and $K$ is an intersection of the Hanner polytope in Remark \ref{HannerP} and the positive orthant. Note that $K$ is an anti-blocking body satisfying  the equality case of inequality \eqref{m-proj}, i.e.,
    \begin{equation} \label{m-proj-eq-2}
        \frac{\left| K\right|\left| P_{E\cap H} K \right| }{\left| P_E K \right|\left| P_H K\right|} = d_{n,2}(i,j).
    \end{equation} 
    Using \eqref{vol_lp_sum_2}, we have
    \begin{equation} \label{m-plun-ruz-lp-lc}
        \begin{split}
           & |A \oplus_p -B \oplus_p -C | 
            \\
            = &\sum_{k=0}^n  \sum_{H\in \GK[n][k]} \left( \frac{\GM[k] \GM[n-k]}{\GM[n]} t^k \left|P_HK\right| \left|P_{H^\perp} (U\oplus_p V)\right|_{n-k} \right).
        \end{split}
    \end{equation}
    Since $V$ is a cube, we have
    \[
        \left|P_{H^\perp}  (U\oplus_p V)\right|_{n-k} 
        = \left|P_{H^\perp} U\oplus_p -P_{H^\perp} V\right|_{n-k}.
    \]
    Using \eqref{vol_lp_sum_2}, we can write $\left|P_{H^\perp} U\oplus_p -P_{H^\perp} V\right|_{n-k}$ as
    \begin{align*}
        \sum_{l=0}^{n-k}  \underset{L \subset H^\perp}{\sum_{L\in \GK[n][l], }} \left(\frac{\GM[l]\GM[n-k-l]}{\GM[n-k]} \left| P_LU\right|_l \left| P_{L^\perp \cap H^\perp} V \right|_{n-k-l} \right).
    \end{align*}
    Note that we only need to consider the case  $L = H^\perp \cap E^\perp$ and $E \cap F \subset H$, because at least one of  $\left| P_LU\right|_l$ and $\left| P_{L^\perp \cap H^\perp} V \right|_{n-k-l}$ is zero in all other cases. We will have only one term left, i.e.,
    \begin{equation} \label{m-plun-ruz-lp-lc2}
        \left|P_{H^\perp} U\oplus_p -P_{H^\perp} V\right|_{n-k} = \frac{\GM[k]\GM[l]}{\GM[n-k]},
    \end{equation}
    where $ l $ is the dimension of $ H^\perp \cap E^\perp$. Then, \eqref{m-plun-ruz-lp-lc} yields together with \eqref{m-plun-ruz-lp-lc2},
    \begin{equation}
        \begin{split}
        |A|&|A \oplus_p -B \oplus_p -C | 
        \\
        &= t^n|K|\sum_{k=i+j-n}^n  \underset{E\cap F \subset H}{\sum_{H\in \GK[n][k], }}
        \left( \beta_3(k,l) t^k\left|P_HK\right| \right),          
        \end{split}
    \end{equation}
    where 
    \[
        \beta_3(k,l) = \frac{\GM[k]\GM[l]\GM[n-k-l]}{\GM[n]}.
    \]
    Similarly, we have 
    \begin{equation} \label{m-plun-ruz-lp-lc3}
    \begin{split}
        &\left| A \oplus_p -B \right| \left| A \oplus_p -C \right|\\
        &= \sum_{m =i}^n \underset{F\subset M}{\sum_{M \in \GK[n][m], }} \sum_{r =j}^n \underset{E\subset R}{\sum_{R \in \GK[n][r], }}
        \left( \beta_4(m,r)
        t^{m+r} \left| P_MK \right| \left| P_RK\right| \right),
    \end{split}
    \end{equation}
    where
    \[  
        \beta_4(m,r) = \frac{\GM[m] \GM[n-m]\GM[r]\GM[n-r]}{\GM[n]^2}.
    \]
    Using \eqref{m-plun-ruz-lp-lc2} and \eqref{m-plun-ruz-lp-lc3} in \eqref{Plun-Ruz-lp}, by dividing the result by \(t^{i+j} \) and taking \( t= 0\), we get 
    \[
        \frac{\GM[i+j-n]\GM[n]}{\GM[i]\GM[j]}|K| |P_{E\cap F} K| \leq b(n,p)  \left| P_F K \right| \left| P_E K\right| .
    \]
    We are finished the proof using \eqref{m-proj-eq-2}.
\end{proof}

\begin{remark}
  An  upper and lower bounds for the constant $b(p,n)$ are provided in Proposition \ref{Plun-Ruz-lp-cal} below. One can also use a similar method as in the proof of Theorem \ref{m-upcn-sum} to generalize Theorem \ref{m-plun-ruz-lp} to the case of $m-$convex bodies: For any anti-blocking bodies $A,B_1,\ldots,B_m \subset \R^n_+$, one has 
    \[
        \left| A \right|^{m-1} \left| A\oplus_p \pm B_1 \oplus_p \ldots \oplus_p \pm B_m \right| \leq \varrho_{n,m} \prod_{i=1}^m \left| A \oplus_p - B_i \right|,
    \]
    where 
    \[
        \varrho_{n,m} =\max_{0\leq \alpha_i \leq n} \,
        \left( d_{n,m} (\alpha_1,\ldots,\alpha_m) \frac{\Gamma\left(1+\frac{N}{q}\right)\Gamma\left(1+\frac{n}{q}\right)^{m-1}}{\prod_{i=1}^m\Gamma\left(1+\frac{\alpha_i}{q}\right)} \right),
    \]
    where $d_{n,m}$ is defined in (\ref{eq:dnm}). Moreover, $\varrho_{n,m}$ is the best possible constant. 

\end{remark}




Next, in Lemma  \ref{RS-l-anti}, we will present a sharp constant for a Roger-Shephard type inequality for $L_p-$difference of a locally anti-blocking body. We will use  a method from  \cite{S-24}. It was proved in \cite{ASS-23} that for locally anti-blocking bodies $K,L$, one has
\[
    K+L = \bigcup_{\delta \in \{ -1,1 \}^n } \left(K_\delta + L_\delta \right),
\]
where $K_\delta = K \cap \delta \R^n_+$. The following is a very nice observation from \cite{S-24}.

\begin{lemma} \label{fact-anti}
    Let $K \subset \mathbb{R}^n$ be locally anti-blocking, and let $E:=\operatorname{span}\left\{e_i: i \in I\right\} $ for some $I \subset[n]$. Let $\tau, \delta \in\{-1,1\}^n$ be two orthant signs, such that $\left.\tau\right|_I=\left.\delta\right|_I$. Then,
    $$
        P_E K_\delta=P_E K_\tau .
    $$
\end{lemma}

\begin{lemma} \label{RS-l-anti}
    Let $K$ be a locally anti-blocking body. Then,
    \[
        | K \oplus_p -K | \leq \kappa_{n,q} |K|
    \]
    where
    \[
        \kappa_{n,q} = \sum_{i=0}^n \binom{n/q}{i/q}^{-1}  \binom{n}{i}^2.
    \]
    Moreover, $\kappa_{n,q}$ is the best possible constant.
\end{lemma}

\begin{proof}
    For any locally anti-blocking bodies $K,L$, we have
    \begin{align*}
        K \oplus_p  L &= \bigcup_{t \in [0,1]} \left((1-t)^{1/q}  K + t^{1/q} L \right)
        \\
        &= \bigcup_{t \in [0,1]} \bigcup_{\delta \in \{-1,1\}^n} \left( (1-t)^{1/q}  K_\delta + t^{1/q} L_\delta \right)
        \\
        &=  \bigcup_{\delta \in \{-1,1\}^n} \bigcup_{t \in [0,1]} \left( (1-t)^{1/q}  K_\delta + t^{1/q} L_\delta \right)
        \\
        &=  \bigcup_{\delta \in \{-1,1\}^n} \left( K_\delta \oplus_p L_\delta \right).
    \end{align*}
    Note that sets  $ K_\delta \oplus_p L_\delta,$ in the above union,  are disjoint up to a set of measure zero, so we have 
    \[
        \left| K \oplus_p - K \right| 
        = \sum_{\delta \in \{ -1,1\} ^n } \left| K_\delta \oplus_p (-K)_\delta \right|.
    \]
    Using \eqref{diff_lp_vol}, since $(-K)_\delta = -(K_{-\delta})$, we have
    \begin{align*}
        \left| K \oplus_p - K \right| 
        &= \sum_{\delta \in \{ -1,1\} ^n } \left| K_\delta \oplus_p (-K)_\delta \right|
        \\
        &= \sum_{\delta \in \{ -1,1\} ^n } \left| K_\delta \oplus_p -(K_{-\delta}) \right|
        \\
        &\leq \sum_{\delta \in \{ -1,1\} ^n } \left| K_\delta \oplus_p K_{-\delta} \right|.
    \end{align*}
    Next,
    \begin{align*}
        \sum_{\delta \in \{ -1,1\} ^n } \left| K_\delta \oplus_p K_{-\delta} \right|
        &= \sum_{\delta \in \{ -1,1\} ^n } \sum_{i=0}^n \sum_{E \in \GK[n][i]} \binom{n/q}{i/q}^{-1}
        \left| P_EK_\delta\right| \left| P_{E^\perp} K_{-\delta} \right|.
    \end{align*}
    Fixing $E = \Span\{e_i: i\in I\}\in \GK[n][i]$, for each $\delta\in \{-1,1\}^n$, we claim that there exists 
    $\tau \in \{ -1,1\}^n$ such that 
    \begin{equation} \label{con-fact-anti}
        P_E K_\delta = P_E K_\tau \text{ and } P_{E^\perp} K_{-\delta} = P_{E^\perp} K_\tau.
    \end{equation}
    Moreover, it is a bijection. 
    Indeed, the map $\tau \longmapsto \delta(E,\tau)$ where $\delta = \delta(E,\tau)$ is unique vector  satisfying
    \[
        \left.\delta\right|_I=\left.\tau\right|_I,\left.\quad \delta\right|_{I^c}=-\left.\tau\right|_{I^c}.
    \]
    Applying Lemma \ref{fact-anti}, we get \eqref{con-fact-anti}.
    Using \eqref{con-fact-anti}, \eqref{lb-RS}, we have
    \begin{align*}
        \left| K \oplus_p -K \right| &\leq \sum_{i=0}^n   \sum_{E \in \GK[n][i]} \binom{n/q}{i/q}^{-1}
        \sum_{\tau \in \{ -1,1\} ^n } \left| P_EK_\tau\right| \left| P_{E^\perp} K_\tau \right|
        \\
        &\leq \sum_{i=0}^n  \sum_{E \in \GK[n][i]} \binom{n/q}{i/q}^{-1}
        \sum_{\tau \in \{ -1,1\} ^n } \binom{n}{i} \left| K_\tau \right|
        \\
        &= \sum_{i=0}^n  \sum_{E \in \GK[n][i]} \binom{n/q}{i/q}^{-1}
        \binom{n}{i} \left| K \right|
        \\
        &= \sum_{i=0}^n  \binom{n/q}{i/q}^{-1}
        \binom{n}{i}^2 \left| K \right|.
    \end{align*}
    Note that the constant is sharp when $K$ is a simplex.
\end{proof}
We note that it follows from Lutwak's Brunn-Minkowski inequality for $L_p-$sums \cite{L-93, LYZ-12} that $| K \oplus_p -K | \geq 2^{n/p} |K|.$ Using   comparison between $L_p-$sums and Minkowski sums together with classical Rogers-Shephard inequality one can get that $\kappa_{n,q} \leq \binom{2n}{n}$.  One can also get this estimate directly using the properties of Binomial coefficients, see  Proposition \ref{kappa-cal} below. 

It is a very natural question to ask for equality cases of Lemma  \ref{RS-l-anti}:\\

\noindent{\bf Conjecture:} The  equality  in Lemma  \ref{RS-l-anti} holds if and  only if $K$ is a simplex.

\section{The case of  unconditional product measures}
In this section, we will extend our result to an unconditional, product measure $\mu = \bigotimes_{i=1}^n \mu_i$ with decreasing density which is $1/n-$concave on 1-unconditional convex bodies.

We note that  the proof of \eqref{BT-int}  works if the volume is replaces by a product measure with decreasing density, i.e., for any compact set $K$, $\lambda-$uniform cover $(\sigma_1,\ldots,\sigma_m)$ of $[n]$, we have
\begin{equation} \label{BT-sm}
    \mu(K)^\lambda \leq \prod_{i=1}^m \mu_{\sigma_i} \left( P_{E_{\sigma_i}^\perp} K \right),
\end{equation}
where $\mu_{\sigma_i} = \underset{k \in \sigma_i}{\bigotimes} \mu_k$. 

The proof of the next theorem directly follows from the proof of Theorem \ref{LM-uni-cov} together with \eqref{BT-sm}, Theorem \ref{uncon-prod-m} and \eqref{Berwald-ine-s}. 
It is important to note that our result applies to the class of 1-unconditional convex bodies, where the measure is $1/n-$concave, as established in Theorem \ref{uncon-prod-m}.
\begin{theorem}
    Let $\sigma \subset [n]$. Fix $m \geq \lambda \geq 1$ and let $(\sigma_1, \ldots, \sigma_m)$ be $\lambda-$uniform cover of $\sigma$. For any 1-unconditional convex body $K \subset \R^n$, one has
    \begin{equation} \label{L-LM-unif-s}
        \mu\left( K \right)^{m-\lambda} \mu_{\sigma^c}\left(P_{E_\sigma} K \right)^\lambda \leq 
        \frac{\prod_{i=1}^m \binom{n- |\sigma_i|}{n-|\sigma|} }{\binom{n}{n-|\sigma|}^{m-\lambda}} \prod_{i=1}^m \mu_{\sigma_i^c}\left( P_{E_{\sigma_i}} K \right),
    \end{equation}
    where $\sigma^c = [n]\smallsetminus \sigma$.
\end{theorem}

\begin{corollary}
    The inequality \eqref{L-LM-unif-s} holds for the class of anti-blocking bodies.
\end{corollary}

The following theorem follows from the proof of Theorem  \ref{L1-sum} and inequality  \eqref{L-LM-unif-s}.

\begin{theorem}\label{L1-sum-gaussian}
    Let \( A,B \subset \R^n_+\) be anti-blocking bodies. Then, for any proper coordinate subspace \( E \in \GK[n][i] \),
    \begin{equation} \label{ratiofp-gaussian}
        \frac{\mu(A)}{\mu_{\sigma_i^c} (P_E A)} +  \frac{\mu ( B)}{\mu_{\sigma_i^c} (P_E B )} \leq r_{n,i} \left(  \frac{\mu( A - B)}{\mu_{\sigma_i^c}( P_E(A - B) )} \right),
    \end{equation}
    where $r_{n,i}$ is defined in Theorem \ref{L1-sum}.
\end{theorem}

\begin{lemma}
    Let $K$ be a locally anti-blocking body. Then, for any $i$,
    \begin{equation} \label{s-shift}
        \mu( S_i(K) ) \leq \mu( K ),
    \end{equation}
\end{lemma}

\begin{proof}
    Without loss of generality, assume that $ i =n$. For each $x \in P_{e_n^\perp} K$, there exist $t_1\leq 0 \leq t_2 $ such that $$
        K \cap (x + \R e_n) = \{ x+ te_n : t_1 \leq t \leq t_2 \}.
    $$ 
    This implies that
    \[
        S_n(K) \cap (x + \R e_n) = \{ x + te_n : 0 \leq t \leq t_2 -t_1 \}.
    \]
    Then,
    \begin{align*}
        \mu_n( K \cap (x + \R e_n) )  &= \int_{t_1}^{t_2} \phi(x_1,\ldots,x_{n-1},t) d \mu_{n} (t)
        \\
        &\geq \int_{0}^{t_2-t_1} \phi(x_1,\ldots,x_{n-1},t) d \mu_{n} (t)
        \\
        &\geq \mu_n (S_i(K) \cap (x + \R e_i)).
    \end{align*}
    Thus,
    \begin{align*}
        \mu (K) &= \int_{P_{e_n^{\perp}} K} \mu_n( K \cap (x + \R e_i) ) d \mu_{n^c}
        \\
        &\geq  \int_{P_{e_n^{\perp}} K} \mu_n (S_i(K) \cap (x + \R e_i)) d \mu_{n^c} 
        \\
        &=\mu(S_n(K)). \qedhere
    \end{align*}
\end{proof} 

\begin{lemma} \label{difference-sum-s}
    Let $A,B \subset \R^n_+$ be anti-blocking bodies. Then,
    \[
        \mu( A+B)  \leq \mu(A-B).
    \]
\end{lemma}

\begin{proof}
    Using \eqref{s-shift}, we have
    \[
        \mu( A -B ) \geq \mu( S_n(A -B) ) \geq \mu( S_n(A) + S_n(-B) ) = \mu( A + S_n(-B) ).
    \]
    Repeat the process by applying $S_1, \ldots, S_{n-1}$, we will get
    \[
        \mu( A -B ) \geq \mu(A + S_1 S_2 \cdots S_n (-B) ).
    \]
    We end the proof here since $S_1 S_2 \cdots S_n (-B) = B$.
\end{proof} 

Next  theorem  can be proved using the  the proof of Theorem \ref{m-upcn-sum} together with Lemma \ref{difference-sum-s} and \eqref{L-LM-unif-s}.

\begin{theorem}
    Let $A,B_1,\ldots,B_m \subset \R^n_+$ be anti-blocking bodies. Then,
    \begin{equation} \label{Plunn-Nec-s}
        \mu(A)^{m-1} \mu(A\pm B_1 \pm B_2\pm \cdots \pm B_m) \leq \zeta_{n,m} \prod_{i=1}^m \mu\left( A - B_i \right),
    \end{equation}
    where $\zeta_{n,m}$ is defined in Corollary \ref{corant}.
\end{theorem}

\begin{remark} We note that  constants in inequalities \eqref{L-LM-unif-s}, \eqref{ratiofp-gaussian} and  \eqref{Plunn-Nec-s} are sharp. Indeed, those inequalities 
are invariant under  multiplication of  the density $\phi$ of $\mu$ by a positive constant. Thus we may assume that $\phi(0)=1$. Next, taking  anti-blocking bodies, belonging to a small enough Euclidean ball centered  at the origin, using the continuity of  $\phi$ we get that the constants in \eqref{L-LM-unif-s}, \eqref{ratiofp-gaussian} and  \eqref{Plunn-Nec-s} can not be better than the constants in the corresponding inequalities for Lebesgue measure.
\end{remark}

\section{Appendix}
In this section, our aim is to present the optimization of the constants along with some more technical calculations.

The first calculation is a specific example for the equality case in \eqref{L-LM-unif}.  
\begin{lemma} \label{Cal-LM}
    The equality in  inequality \eqref{L-LM-unif} can be reached by using the following Hanner polytope,
    \[
        K = \Conv\left\{ \sum_{i \in \sigma} [-e_i,e_i], \sum_{i \not\in \sigma } [-e_i,e_i] \right\} ,
    \]
    where $\sigma_1, \ldots, \sigma_n$ is a $\lambda-$uniform cover of $\sigma$.
\end{lemma}
\begin{proof}
    Without loss of generality, assume that $ \sigma = \{1, \ldots , |\sigma|\} $. Using (\ref{eq:volint}) we get 
    \begin{align*}
        \left| K \right| 
        &= \frac{1}{n!} \int_{\R^n} e^{-\| x\|_K} dx 
        \\
        &= \frac{1}{n!} \int_{\R^n} e^{-\| x_1\|_{B_\infty^{|\sigma|}}}  dx_1  \int_{\R^n} e^{-\| x_2\|_{B_\infty^{n-|\sigma|}}}  dx_2 
        \\
        &= \binom{n}{n-|\sigma|}^{-1} \left| B_\infty^{|\sigma|} \right|\left| B_\infty^{n-|\sigma|} \right|  = \binom{n}{n-|\sigma|}^{-1}2^{n}.
    \end{align*}
    Similarly, we have
    \[
        \left| P_{E_{\sigma_i}} K \right| = \binom{n-|\sigma_i|}{n-|\sigma|}^{-1}2^{n-|\sigma_i|}
   \mbox{ and } 
        \left| P_{E_{\sigma}} K \right| = 2^{n-|\sigma|}.
    \]
    Then,
    \[
        \frac{\left| K \right|^{m-\lambda} \left|P_{E_\sigma} K \right|^\lambda}{\prod_{i=1}^m \left| P_{E_{\sigma_i}} K \right|} = \frac{\prod_{i=1}^m \binom{n-|\sigma_i|}{n-|\sigma|} }{\binom{n}{n-|\sigma|}^{m-\lambda}}. \qedhere
    \]
\end{proof}
We refer to \cite{AA-65}, pp 254-259, for the basic properties of the digamma function.
The digamma function is the logarithm derivative of the Gamma function. In particular, if 
\[
    \psi (z) = \frac{d }{dz} \ln \Gamma(z) = \frac{\Gamma'(z)}{\Gamma (z)},
\]
then
\begin{equation} \label{digamma-recursion}
    \psi (z+1) =  \psi (z) + \frac{1}{z},
\end{equation}
and, for any $ z \neq -1,-2, \ldots $, one has
\begin{equation} \label{digamma-series}
    \psi(z) = -\gamma+\sum_{t=1}^{\infty}\left(\frac{1}{t}-\frac{1}{t+z}\right),
\end{equation}
where $\gamma$ is known as the Euler constant.

\begin{lemma} \label{cn-calculation}
    The optimization of the constant in \eqref{m-upcn-sum-wc} when $ m=2$ is divided into three cases:
   \[
       \zeta_{n,2} =
       \begin{cases}
           \displaystyle \frac{(2k)!^3}{(3k)!k!^3}; & n =3k \\
           \displaystyle \frac{(2k+1)^2}{(k+1)(3k+1)}\frac{(2k)!^3}{(3k)!k!^3}; & n=3k+1 \\
           \displaystyle \frac{2(2k+1)^3}{(k+1)(3k+1)(3k+2)}\frac{(2k)!^3}{(3k)!k!^3}; & n =3k+2.
       \end{cases}
   \]
\end{lemma}

\begin{proof}
    Recall the definition $\zeta_{n,2}$ for $ m=2$,
    $$
        \zeta_{n,2}= \max d_{n,2} (i,j) = \max \frac{\binom{i}{i+j-n }\binom{j}{i+j-n }}{\binom{n}{i+j-n }},
    $$
    where the maximum is taken over all non-negative integers $i,j \leq n$ and 
    $i+j-n \geq 0.$ Let's fix $i+j=l$ and, without loss of generality, assume that $i\geq j$. 
    We define the function $f_l$ on $[n/2,n]$ as follows:
    \begin{align*}
        f_l(x) 
        &:= \log\left[ \frac{\Gamma(x+1)\Gamma(l-x+1)\Gamma(2n-l+1)}{\Gamma(n+1)\Gamma(l-n+1)\Gamma(n-x+1)\Gamma(n+x-l+1)}\right].
    \end{align*}
    Note that for any integer $i$, $\log d(i,l-i) = f_l(i)$. 
    Then,
    $$
        f'_l(x)= \psi(x+1)-\psi(l-x+1)+\psi(n-x+1)-\psi(n+x-l+1).
    $$
    Using  \eqref{digamma-recursion}, we have
    $$  
        \psi(x+1)-\psi(n+x-l+1)
        =\sum_{t=1}^{l-n} \frac{1}{t+x+n-l} 
    $$
    and 
    $$
        -\psi(l-x+1)+\psi(n-x+1)= -\sum_{t=1}^{l-n} \frac{1}{t-x+n}.
    $$
    Thus,
    $$
        f''_l(x)=-\sum_{t=1}^{n-l} \left( \frac{1}{(t+x+n-l)^2}+\frac{1}{(t-x+n)^2} \right)< 0.
    $$
    Therefore, the function $ f_l $ is concave. Since $f_l(x) = f_l(l-x)$,
    one can deduce that 
    $f_l$ is reached the maximum at $x=l/2$.
    Similarly, we define a function $g$ on $[n/2,n]$,
    $$ 
        g(x)
        :=
        \log\left[ \frac{\Gamma^2(x+1)\Gamma(2n-2x+1)}{\Gamma(n+1)\Gamma(2x-n+1)\Gamma^2(n-x+1)}\right].
    $$
    Note that for any integer $i$, $d_{n,2} (i,i) = g (i)$. Using \eqref{digamma-series}, we have
    \begin{align*}
        g'(x)&=2\psi(x+1)-2\psi(2n-2x+1)-2\psi(2x-n+1)+2\psi(n-x+1)\\
        &= 2\sum_{t\geq 1} \left(\frac{1}{2n-2x+t}-\frac{1}{x+t}+\frac{1}{2x-n+t}-\frac{1}{n-x+t}\right)\\
        &=2(3x-2n)\sum_{t\geq 1} \left(\frac{1}{(2n-2x+t)(x+t)}-\frac{1}{(2x-n+t)(n-x+t)} \right)
        \\
        &= 2(3x-2n)(x-n) \sum_{t\geq 1} \frac{n+2t}{(2n-2x+t)(x+t)(2x-n+t)(n-x+t)}. 
    \end{align*}
    Since the last summation is positive, we have that $g'(x) > 0$ for $x < 2n/3$ and $g'(x) <0$ for $2n/3 < x $. Therefore, the maximum of $g$ reaches at $ x_{\max}= 2n/3$.

\begin{itemize}
    \item If $n=3k$, then$ \underset{n/2 <i<n}{\max} g(i)= g(2k).$ We get
    $$
        \zeta_{3k,2}=\frac{(2k)!^3}{(3k)!k!^3}.
    $$
    \item If $n=3k+1,$ then $ x_{\max} =2k+2/3$. We get
    \begin{align*}
        \zeta_{3k+1,2} &= \max \{g(2k), g(2k+1), f_{4k+1}(2k+1)\} 
        \\
        &= \frac{(2k+1)^2}{(k+1)(3k+1)}\frac{(2k)!^3}{(3k)!k!^3}.
    \end{align*}
    %
    \item If $n=3k+2,$ then $ x_{\max} =2k+4/3$. We get
    \begin{align*}
        \zeta_{3k+2,2} &= \max \{g(2k+1), g(2k+2), f_{4k+3}(2k+2)\} 
        \\
        &= \frac{2(2k+1)^3}{(k+1)(3k+1)(3k+2)}\frac{(2k)!^3}{(3k)!k!^3}. \qedhere
    \end{align*}
\end{itemize}
\end{proof}

\begin{remark}
    Using Stirling's approximation, we have $\zeta_{n,2} \sim \sqrt{\frac{4}{\pi n}} \left( \frac{4}{3} \right)^n$.
\end{remark}

\begin{prop} \label{rn-calculation}
    The optimization of the constant in \eqref{ratiofp} is divided into two cases:
    \[
       r_{n,i} =
       \begin{cases}
           \displaystyle \frac{(2k)!(n-k)!^2}{k!^2(n-2k)!n!}; & i =2k \\
           \displaystyle \frac{(2k+1)(n-2k)}{(k+1)(n-k)} \frac{(2k)!(n-k)!^2}{k!^2(n-2k)!n!}; & i=2k+1.
       \end{cases}
   \]
\end{prop}

\begin{proof}
    Recall that,
    \[
        r_{n,i} = \max_j d_{n,2}(i,j) = \max \frac{\binom{i}{i+j-n} \binom{j}{i+j-n }}{\binom{n}{i+j-n }},
    \]
    where the maximum is taken over all non-negative integers $j \leq n$ such that $i+j \geq n+1$. We define the function $g$ on $[n-i,n]$ as follow: 
    \[
        g(x) = \log \left[ 
            \frac{\Gamma(x+1) \Gamma ( i+1) \Gamma(2n-i-x+1)}{\Gamma (n-x+1) \Gamma (n-i+1)\Gamma(i+x-n+1) \Gamma (n+1)}
        \right].
    \]
    Note that for any integer $j$, $\log d_{n,2}(i,j) = g(j)$. By using \eqref{digamma-series}, $g'(x)$ can be expanded as 
    \[
        (2x-2n+i)(i-n)\sum_{t \geq 1} \frac{2t+n}{(2n-i-x+t)(x+t) (n-x+t)(i+x-n+t)}.
    \]
    Since the summation is positive, we have $g'(x) <0$ for $x > n-i/2$ and $g'(x) > 0 $ for $ x < n-i/2$. So, $g$ reaches the maximum at $ n-i/2 $.
    \begin{itemize}
        \item If $ i =2k$, then
        \[
            r_{n,2k} = \frac{(2k)!(n-k)!^2}{k!^2(n-2k)!n!}.
        \]
        \item If $ i = 2k+1$, then
        \begin{align*}
            r_{n,2k+1} &= \max \{ g(n-k-1), g(n-k) \}
            \\
            &= \frac{(2k+1)(n-2k)}{(k+1)(n-k)} \frac{(2k)!(n-k)!^2}{k!^2(n-2k)!n!}. \qedhere
        \end{align*}
    \end{itemize}
\end{proof}

\begin{prop} \label{Plun-Ruz-lp-cal}
 Let $b(n,p)$ be the     constant from Theorem \ref{m-plun-ruz-lp},  then 
    \[
        \frac{2}{e\sqrt{\pi n}} \left( \frac{4}{3}\right)^n \leq b(n,p) \leq \frac{2^{n+1/2}}{\sqrt{\pi n}} .
    \]
\end{prop}

\begin{proof}
    Recall that
    \[
        b(n,p) = \underset{0\leq i,j \leq n}{\max }\,
        \left(d_{n,2} (i,j) \frac{\Gamma\left(1+\frac{i+j-n}{q}\right)\Gamma\left(1+\frac{n}{q}\right)}{\Gamma\left(1+\frac{i}{q}\right)\Gamma\left(1+\frac{j}{q}\right)}\right).
    \]
    Fix $n,i,j$ such that $k:= i+j-n \geq 0$, and  define a function $f:[1,\infty) \to \R$ such that
    \[
        f(q) = \log \left[ \frac{\Gamma\left(1+\frac{k}{q}\right)\Gamma\left(1+\frac{n}{q}\right)}{\Gamma\left(1+\frac{i}{q}\right)\Gamma\left(1+\frac{j}{q}\right)} \right].
    \]
    Then,
    \[
        f'(q) = -\frac{k}{q^2}\psi \left( 1+ \frac{k}{q} \right) -\frac{n}{q^2}\psi \left( 1+ \frac{n}{q} \right) +\frac{i}{q^2}\psi \left( 1+ \frac{i}{q} \right) +\frac{j}{q^2}\psi \left( 1+ \frac{j}{q} \right).
    \]    
    Using \eqref{digamma-series}, we have
    \[
    f'(q) = (n-i)\sum_{t\geq 1} \left( \frac{t}{(qt+n)(qt+i)} - \frac{t}{(qt+j)(qt+k)} \right).
    \]
    Then, $f$ is decreasing on $[1,\infty)$ as a function of $q$, which implies that $b(n,p)$ is increasing as $p$ increase. We have
    \[
      \underset{0\leq i,j \leq n}{\max }\,
        d_{n,2} (i,j) = b(n,1) \leq b(n,p) \leq b(n,\infty) = \underset{0\leq i,j \leq n}{\max }\,
        \frac{(2n-i-j)!}{(n-i)!(n-j)!}.
    \]
    It was provided in \cite{FMZ-24} by using Stirling approximation that
    \[
        \underset{0\leq i,j \leq n}{\max } d_{n,2} (i,j) \geq \frac{2}{e\sqrt{\pi n}}\left(\frac{4}{3}\right)^n.
    \]
    Fix $l = i+j$, we define function $g_l$ by
    \[
        g_l(x) = \log \left[ \frac{\Gamma(1+2n-l)}{\Gamma(1+n-l+x)\Gamma(1+n-x)} \right].
    \]
    Taking the derivative and using \eqref{digamma-series}, we get
    \[
        g_l'(x) =\sum_{t \geq 0} \frac{l-2x}{(n-l+x+t)(t+n-x)}.
    \]
    The function $g$ reaches the maximum when $ x = l/2$. It implies that 
    \[
        \underset{0\leq i,j \leq n}{\max }\,
        \frac{(2n-i-j)!}{(n-i)!(n-j)!} \leq \underset{0\leq i,j \leq n}{\max }\,
        \frac{(2n-2i)!}{((n-i)!)^2} \leq \binom{n}{n/2}.
    \]
    By using the Stirling's approximation, we have 
    \[
        \binom{n}{n/2} \leq \frac{2^{n+1/2}}{\sqrt{\pi n}}. \qedhere
    \]
\end{proof}

\begin{prop} \label{kappa-cal} Let  $\kappa_{n,q}$ be the constant from  Lemma \ref{RS-l-anti}, then
    \[
        2^n \leq \kappa_{n,q} \leq \binom{2n}{n}.
    \]
\end{prop}

\begin{proof}
    Recall that,
    \[
        \kappa_{n,q} = \sum_{i=0}^n \binom{n/q}{i/q}^{-1}  \binom{n}{i}^2.
    \]
    Fix $n$ and $i$, we define a function $f$ on $[1,\infty)$ as follows:
    \[
        f(q) = \log \binom{n/q}{i/q} 
        = \log \left[\frac{\Gamma \left(\frac{n}{q}+1 \right) 
        }{\Gamma \left(\frac{i}{q}+1 \right) \Gamma \left(\frac{n-i}{q}+1 \right)  } \right].
    \]
    Then,
    \[
        f'(q) = -\frac{n}{q^2}\psi \left(\frac{n}{q} +1\right) + \frac{i}{q^2}\psi \left(\frac{i}{q} +1\right) +\frac{n-i}{q^2}\psi \left(\frac{n-i}{q} +1\right).
    \]
    Using \eqref{digamma-series}, we have
    \[
        f'(q) = -\frac{i(n-i)}{q} \sum_{t \geq 1} \left(\frac{1}{(tq+i)(tq+n)} + \frac{1}{(tq+n-i)(tq+n)} \right).
    \]
    Then, $f$ is decreasing on $[1,\infty]$ which implies that $\kappa_{n,q}$ is increasing as $q$ increase. Hence,
    \[
        2^n = \kappa_{n,1} \leq \kappa_{n,q} \leq \kappa_{n,\infty} = \binom{2n}{n}. \qedhere
    \]
\end{proof} 


\end{document}